\documentclass{article}
\usepackage[T1]{fontenc}
\usepackage{graphicx}
\usepackage{amsmath,amssymb}
\usepackage{thm-restate}

\usepackage[colorlinks=true,bookmarks=false]{hyperref}
\usepackage{xcolor}
\definecolor{darkblue}{rgb}{0,0,0.45}
\definecolor{darkred}{rgb}{0.6,0,0}
\definecolor{darkgreen}{rgb}{0.13,0.5,0}
\hypersetup{colorlinks, linkcolor=darkblue, citecolor=darkgreen,urlcolor=darkblue}

\newcommand{\Dbar}{\overline D}

\newcommand{\fun}{\mathrm{fun}}
\newcommand{\sd}{\mathrm{sd}}
\newcommand{\sym}{\mathrm{\Delta}}

\newcommand{\N}{\mathbb{N}}
\newcommand{\R}{\mathbb{R}}
\newcommand{\E}{\mathbb{E}}

\newcommand{\vc}{\mathrm{vc}}
\newcommand{\cI}{\mathcal{I}}

\newcommand{\INT}{\mathrm{INT}}
\newcommand{\CA}{\mathrm{CA}}

\newcommand{\calG}{\mathcal{G}}

\newcommand{\CellSpace}{\rule{0pt}{14pt} \rule[-9pt]{0pt}{3pt}} 
\newcommand{\SmallCellSpace}{\rule{0pt}{12pt}} 

\usepackage{amsthm, amsmath, amssymb}
\usepackage{thmtools}
\usepackage{authblk} 

\theoremstyle{plain}
\newtheorem{theorem}{Theorem}
\newtheorem{lemma}[theorem]{Lemma}
\newtheorem{corollary}[theorem]{Corollary}
\newtheorem{proposition}[theorem]{Proposition}
\newtheorem{question}{Question}

\theoremstyle{definition}
\newtheorem{observation}[theorem]{Observation}

\newtheorem{remark}[theorem]{Remark}
\newtheorem*{remark*}{Remark}

\begin{document}
\title{Bounds on Functionality and Symmetric Difference -- Two Intriguing Graph Parameters\footnote{The conference version of this paper was published in the proceedings of The 49th International Workshop on Graph-Theoretic Concepts in Computer Science, WG 2023~\cite{WG}.}}

%
%
\author[1]{Pavel Dvořák\thanks{Supported by Czech Science Foundation GAČR grant 22-14872O.}}
\author[2]{Lukáš Folwarczný\thanks{Supported by Czech Science Foundation GAČR grant 19-27871X.}}
\author[3]{Michal Opler\thanks{Supported by Czech Science Foundation GAČR grant 24-12046S.}}
\author[2]{Pavel Pudlák\thanks{Supported by Czech Science Foundation GAČR grant 25-16311S.}}
\author[1]{Robert Šámal\thanks{Partially supported by grant 25-16627S of the Czech Science Foundation. 
  This project has received funding from the European Research Council (ERC) under the European Union’s Horizon 2020
  research and innovation programme (grant agreement No 810115).}}
\author[1]{Tung Anh Vu\thanks{Partially supported by Charles Univ. project UNCE 24/SCI/008, by the ERC-CZ project LL2406 of the Ministry of Education of Czech Republic, and by project 24-10306S of Czech Science Foundation GAČR.}}
%
\affil[1]{Faculty of Mathematics and Physics, Charles University, Prague, Czech Republic}
\affil[2]{Institute of Mathematics, Czech Academy of Sciences, Prague, Czech Republic}
\affil[3]{Faculty of Information Technology, Czech Technical University, Prague, Czech Republic}

\date{}
\maketitle              
\begin{abstract}
\emph{Functionality} ($\fun$) is a graph parameter that generalizes graph degeneracy defined by Alecu et al. [JCTB, 2021].
They research the relation of functionality to many other graphs
parameters (tree-width, clique-width, VC-dimension, etc.).
Extending their research, we completely characterize the functionality of random graph~$G(n,p)$ for all possible $p$.
We provide matching (up to a constant factor) lower and upper bound for a large range of $p$.
If $p \leq 1/{n^{\frac{1}{2} - \varepsilon}}$ or $p \geq 1 - 1/{n^{\frac{1}{2} - \varepsilon}}$ where $\varepsilon$ is a small constant, we provide lower and upper bounds that differ in a $O(\ln^2 n)$ factor, moreover the gap is at most $O(\ln n)$ if $p \leq 1/\sqrt{n}$ or $p \geq 1 - 1/\sqrt{n}$.
It follows from our bounds for $G(n,p)$, that the maximum functionality (roughly $\sqrt{n}$) is achieved for $p \approx 1/\sqrt{n}$.
We complement this by showing that every graph~$G$ on~$n$ vertices have $\fun(G) \le O(\sqrt{ n \ln n})$ and we give a nearly matching $\Omega(\sqrt{n})$-lower bound provided by incident graphs of projective planes.
Previously known lower bounds for functionality were only logarithmic in the number of vertices.

Further, we study a related graph parameter \emph{symmetric difference} ($\sd$),
the minimum of $|N(u) \sym N(v)|$ over all pairs of vertices of the ``worst possible''
induced subgraph.
It was observed by Alecu et al. that $\fun(G) \le \sd(G)+1$ for every graph~$G$.
They asked whether the functionality of interval graphs is bounded.
Recently, Dallard et al. [RiM, 2024]  answered this positively and they constructed an interval graph $G$ with $\sd(G) = \Theta(\sqrt[4]{n})$ (even though they did not mention the explicit bound), i.e., they separate the functionality and symmetric difference of interval graphs.
We show that $\sd$ of interval graphs is at most $O(\sqrt[3]{n})$ and we provide a different example of an interval graph $G$ with $\sd(G) = \Theta(\sqrt[4]{n})$.
Further, we show that $\sd$ of circular arc graphs is $\Theta(\sqrt{n})$.

\end{abstract}

\section{Introduction}
\label{sec:Intro}
Let $G = (V,E)$ be a graph and $v \in V$ be a vertex.
An adjacency matrix $A_G$ of~$G$ is a $0$-$1$ matrix such that its rows and columns are indexed by vertices of~$G$ and $A[u,v] = 1$ if and only if $u$ and $v$ are connected by an edge.
Now, we define the functionality and symmetric difference of a graph -- two principal notions of this paper -- as introduced by Alecu et al.~\cite{functionality}, 
and implicitly also by Atminas et al.~\cite{implicit_representation}. 

A vertex $v$ of a graph $G = (V,E)$ is a \emph{function} of vertices $u_1,\dots,u_k \in V$ (different from $v$) if there exists a boolean function $f$ of $k$ variables such that for any vertex 
$w \in V \setminus \{v,u_1,\dots,u_k\}$ it holds that $A_G[v,w] = f\bigl(A_G[u_1,w],\dots,A_G[u_k,w]\bigr)$.
Informally, we can determine if $v$ and $w$ are connected from the adjacencies of $w$ with the $u_i$'s.
The \emph{functionality} $\fun_G(v)$ of a vertex $v$ in $G$ is the minimum $k$ such that $v$ is a function of $k$ vertices of $G$.
We drop the subscript and write just $\fun(v)$ if the graph $G$ is clear from the context.
Then, the \emph{functionality} $\fun(G)$ of a graph $G$ is defined as 
\[
\fun(G) = \max_{H \subseteq G} \min_{v \in V(H)} \fun_H(v),
\]
where the maximum is taken over all induced subgraphs $H$ of $G$.

It is observed in~\cite{implicit_representation} that if $\fun(G) \le k$ then we can encode $G$ using
$n (2^k + (k+1) \log n)$ bits, where $n$ is the number of vertices of $G$. 
Thus, if every graph $G$ in some graph class~$\calG$ has 
bounded functionality then $\calG$ contains at most $2^{O(n \log n)}$ 
graphs on $n$ vertices. 
Such classes are said to be of factorial growth~\cite{factorial_growth} and include diverse classes of practical importance (interval graphs, line graphs, forests, planar graphs, and more generally all proper minor-closed classes). 
Thus, Alecu et al.~\cite{functionality} introduce functionality as a tool to study graph classes of factorial growth, 
and the related Implicit graph conjecture (although this conjecture was recently 
disproved \cite{hatami_igc_false}). This was also our original motivation. 
Moreover, functionality is a natural generalization of the graph degeneracy, as the degree of a vertex $v$ is a trivial upper bound for the functionality of $v$. 
Thus, it deserves a study for its own sake. 

Alecu et al.~\cite{functionality} research the relation of functionality to many other graph 
parameters: in particular they provide a linear upper bound in terms of clique-width 
and a lower bound in terms of some function of VC-dimension. They also give a lower bound for 
the functionality of the hypercube that is linear in dimension (i.e., logarithmic in the number of vertices). 

Another parameter related to functionality is the symmetric difference.
Given two vertices $u, v$ of $G$, let $\sd_G(u,v)$ (or just $\sd(u,v)$ when the graph is clear from the context) be the number of vertices different from $u$ and $v$ that are adjacent to exactly one of $u$ and $v$.
The \emph{symmetric difference} $\sd(G)$ of a graph $G$ is defined as
\[
\max_{H \subseteq G} \min_{u,v \in V(H)} \sd_H(u, v),
\]
where the maximum is again taken over all induced subgraphs.
The set of neighbors of $v$ in $G$ is denoted by $N_G(v)$, we omit the subscript if the graph is clear from the context.
Thus, we may view $\sd(u,v)$ as the size of the set $(N(u) \sym N(v)) \setminus \{u,v\}$, which explains the term ``symmetric difference''. 
It is noted by Alecu et al.~\cite{functionality} that $\fun(G) \leq \sd(G) + 1$.
However,  there is no lower bound in terms of $\sd$ as there are graphs of bounded functionality and polynomial symmetric difference -- for example the interval graphs.
This was shown by Theorem~5.2 of Dallard et al.~\cite{box_graphs} and by our Theorem~\ref{thm:IntervalLB}.


%

\subsection{Our Results}

In this paper, we show several lower and upper bounds for functionality and symmetric difference of various graph classes.
Our main technical result is characterizing the functionality of random graphs $G(n,p)$ for all possible $p$.
Note that $\fun(G) = \fun(\bar{G})$, where $\bar{G}$ is the complement of $G$.
Since $G(n,1-p)$ is the complement of $G(n, p)$, we state our bounds only for $p \leq \frac{1}{2}$.
Our lower and upper bound match up to a constant factor for a large range of $p$.
Only for small values of $p$ our lower and upper bound differ by a $\ln n$ or $\ln^2 n$ factor.
The bounds are presented in the following table, where $0 < \varepsilon < \frac{1}{2}$ is a small constant, and $O(\cdot)$ with $\Omega(\cdot)$ hide universal constants even independent on $\varepsilon$.
All bounds hold for $G(n,p)$ almost surely, i.e., with probability at least $1 - o(1)$.

\begin{center}
\begin{tabular}{c|c|c}
$p$ & A. S. Lower Bound  (Thm.~\ref{thm:RandomLowerBoundFinal}) & A. S. Upper Bound (Cor.~\ref{cor:RandomUpperBound})  \\
\hline

$\left[ \frac{1}{n^{\frac{1}{2} - \varepsilon}} , \frac{1}{2} \right]$ & $\Omega\left(\frac{\varepsilon}{p} \ln n \right)$ & $O\left(\frac{1}{p} \ln n \right)$ \CellSpace \\
\hline
$\left[ \frac{1}{\sqrt{n}}, \frac{1}{n^{\frac{1}{2} - \varepsilon}} \right]$ & $\Omega\left(\frac{1}{p\ln n} \right)$  & $O\left(\frac{1}{p} \ln n \right)$ \CellSpace \\
\hline
$\left[ \frac{6 \ln n}{n}, \frac{1}{\sqrt{n}}\right]$ & $\Omega\left(\frac{pn}{\ln n} \right)$   & $ 2pn$ \CellSpace \\
\hline
$\left(0,\frac{6 \ln n}{n}\right]$ & --- & $ 6e^2 \ln n$ \SmallCellSpace \\

\end{tabular}
\end{center}

As far as we know, there were only logarithmic lower bounds for functionality~\cite{implicit_representation,functionality,box_graphs}.
Thus, it would have been possible that the functionality is at most logarithmic (similarly to the VC-dimension~\cite{vc_dimension}).
However, it is clear from the table that the maximum functionality of $G(n,p)$ is achieved when $p \approx \frac{1}{\sqrt{n}}$, for such $p$ holds that $\fun(G(n,p))$ is roughly $\sqrt{n}$.
Further, we give an explicit construction of graphs of functionality $\Theta(\sqrt{n})$.
In particular, we show that the functionality of the incidence graph of a finite projective plane of order $k$ is exactly $k + 1$, i.e., roughly $\sqrt{n}$.
We complement this result with an almost matching upper bound that functionality of any graph is at most $O(\sqrt{n \ln n})$.
Thus, the maximum functionality of random graphs (with the right value of $p$) and all graphs is roughly the same.

Further, we study the symmetric difference parameter for two classes of intersection graphs.
The intersection graph of a family of sets ${\cal F} = \{S_1,\dots,S_n\}$ is a graph $G = (V,E)$ where $V = \{v_1,\dots,v_n\}$ and two vertices $v_i$ and $v_j$ are connected if and only if the corresponding sets $S_i$ and $S_j$ of ${\cal F}$ intersect.
An interval graph is an intersection graph of $n$ intervals on a real line.
A circular arc graph is an intersection graph of $n$ arcs of a circle.
We let $\INT_n$ denote the family of all intersection graphs with $n$ vertices, $\INT$ the family of all interval graphs.
In the same vein, we define $\CA_n$ and $\CA$ for circular arc graphs.

We show that the symmetric difference of circular arc graphs is $\Theta(\sqrt{n})$, i.e., we prove that any circular arc graph has symmetric difference at most $O(\sqrt{n})$ and we present a circular arc graph of symmetric difference $\Omega(\sqrt{n})$.
Recently, it was shown that interval graphs have bounded functionality and unbounded symmetric difference~\cite{box_graphs}.
Even though it was not explicitly mentioned, the construction given by Dallard et al.~\cite{box_graphs} leads to the lower bound $\Omega(\sqrt[4]{n})$ for the symmetric difference of interval graphs.
We independently came up with a different construction leading to the same lower bound, however, the analysis of our construction is simpler than the one from the previous work.
For interval graphs, we also present the upper bound $O(\sqrt[3]{n})$ for symmetric difference.
Thus, we are leaving a gap between the lower and upper bound.
However, we show the symmetric difference of interval graphs is polynomial and strictly smaller than symmetric difference of circular arc graphs.

\paragraph*{Technique for the Random Graphs}
Our main technical contribution is the bounds for the functionality of the random graphs.
The idea of the lower bound is the following.
Say, we want to prove that the functionality of $G(n,p)$ is larger than $k = \frac{c_1}{p}\cdot\ln n$ with high probability, where $c_1$ is sufficiently small constant.
First, we fix $k + 1$ vertices $v, u_1,\dots,u_k$ and we bound the probability that $v$ is a function of $u_1,\dots,u_k$.
Let $B$ be the set of vertices that are not connected to any of the vertices $u_1,\dots,u_k$.
The expected size of $B$ is $n \cdot (1 - p)^k \geq \sqrt{n}$ if $c_1$ is small enough.
If $v$ were a function of $u_1,\dots,u_k$ then all vertices of $B$ would be connected to $v$ or to none of them, this would happen only with the probability at most $2\cdot(1-p)^{\sqrt{n}}$ (as we suppose that $p \leq \frac{1}{2}$).
Thus, if $c_1$ is sufficiently small and $p$ is roughly larger than $\ln n / \sqrt[4]{n}$ then $v$ is a function of $u_1,\dots,u_k$ only with small probability.
Moreover, this probability is so small that we can use the union bound over all $(k+1)$-tuples $v,u_1,\dots,u_k$ to prove that there is a vertex $v$ of functionality $k$ in $G(n,p)$ only with small probability.
Since $\fun(G)$ is the maximum of $\min_v \fun_H(v)$ over all induced subgraphs $H$ of~$G$, we have a lower bound for the functionality of random graph that holds with high probability.
These ideas with more careful calculations yield the following lower bound for the functionality of random graphs.

\begin{restatable}{theorem}{RandomLowerBoundFinal}
\label{thm:RandomLowerBoundFinal}
Let $\varepsilon > 0$ be an arbtirarily small constant. Then, with a probability at least $1 - o(1)$, it holds that
\[
\fun\bigl(G(n,p)\bigr) \geq
\begin{cases}
\Omega\left(\frac{\varepsilon}{p} \ln n \right) & \text{if $\frac{1}{n^{\frac{1}{2} - \varepsilon}} \leq p \leq \frac{1}{2}$}, \\
\Omega\left(\frac{1}{p \ln n}\right) & \text{if $\frac{1}{\sqrt{n}} \leq p \leq \frac{1}{n^{\frac{1}{2} - \varepsilon}}$}, \\
\Omega\left(\frac{pn}{\ln n}\right) & \text{if $\frac{\ln n}{n} \leq p \leq \frac{1}{\sqrt{n}}$,}
\end{cases}
\]
where $\Omega(\cdot)$'s hide constants independent on $\varepsilon$.
\end{restatable}

For the upper bound, we can try to use a similar idea.
Again, let $k = \frac{c_2}{p}\cdot \ln n$, now with a large enough constant $c_2$.
We fix $k$ vertices $D = \{u_1,\dots,u_k\}$.
If $c_2$ is large enough, then with high probability (roughly $1 - \frac{1}{n^c}$ where the constant $c$ depends on $c_2$) each vertex not in $D$ has a unique neighborhood in $D$.
We call such set a \emph{distinguishing set}.
Formally, given a graph $G$, a set $D \subseteq V(G)$ is distinguishing if there are no vertices $u, v \in V(G) \setminus D$ such that
$N(u) \cap D = N(v) \cap D$.
Thus, each vertex not in $D$ is actually a function of $D$ as we can uniquely determine each vertex not in $D$ by its neighborhood in $D$.
Therefore, if each induced subgraph $H$ of $G$ has a distinguishing set $D \subseteq V(H)$ of size at most $k$, then $\fun(G) \leq k$.

However, this probability bound ($\sim 1 - \frac{1}{n^c}$) is not large enough to use it in the union bound over all induced subgraphs of $G$ to prove that there is a distinguishing set in each induced subgraph of $G$ and thus, $\fun(G) \leq k$.
To improve the bound for the probability, we use the method of alterations.
We will prove that with a very high probability, there are two disjoint sets $D_G$ and $E_G$ of vertices of $G = G(n,p)$ of size at most $k$, such that each vertex not in $D_G \cup E_G$ has a unique neighborhood in $D_G$.
Thus, each vertex of $G$ not in $D_G \cup E_G$ is a function of at most $2k$ vertices.
The probability of this event is so high that we can use the union bound to prove that every induced subgraph $H$ of $G(n,p)$ has such a pair of sets $E_H,D_H \subseteq V(H)$.
Therefore, we can conclude that with high probability $\fun(G) \leq 2k$.
Formally, we will prove the following theorem about distinguishing sets of random graphs.

\begin{restatable}{theorem}{RandomUpperBound}
\label{thm:RandomUpperBound}
  Let $\frac{1}{\sqrt{n}} \leq p \leq \frac{1}{2}$ and $c$ be a sufficiently large constant.
  Then, with probability at least $1 - o(1)$ every induced subgraph of~$G(n,p)$ has a distinguishing set of size at most~$\frac{c}{p}\cdot \ln n$.
\end{restatable}

For the remaining range of $p$, we use an upper bound for the maximum degree of $G(n,p)$.
It is clear that for any graph $G$ holds that $\fun(G)$ is smaller than the maximum degree $\Delta(G)$ of~$G$.
By standard technique, one can easily verify the following bounds for the maximum degree of random graphs, see for example the book by Frieze and Karoński~\cite{max_degree_random}.

\begin{lemma}
\label{lem:RandomMaxDegree}
 With probability at least $1 - o(1)$, it holds that
 \[
\Delta\bigl(G(n,p)\bigr) \leq
\begin{cases}
2pn & \text{if $\frac{6 \ln n}{n} < p \leq \frac{1}{\sqrt{n}}$,} \\
6e^2 \ln n & \text{if $0 < p \leq \frac{6 \ln n}{n}$.}
\end{cases}
\]
\end{lemma}

By discussion above, we have the following corollary of Theorem~\ref{thm:RandomUpperBound} and Lemma~\ref{lem:RandomMaxDegree} that summarizes our upper bound of the functionality of random graphs.

\begin{corollary}
\label{cor:RandomUpperBound}
 With probability at least $1 - o(1)$, it holds that
\[
\fun\bigl(G(n,p)\bigr) \leq
\begin{cases}
O\left(\frac{1}{p} \ln n \right) & \text{if $\frac{1}{\sqrt{n}} < p \leq \frac{1}{2}$,} \\
2pn & \text{if $\frac{6 \ln n}{n} < p \leq \frac{1}{\sqrt{n}}$,} \\
6e^2 \ln n & \text{if $0 < p \leq \frac{6 \ln n}{n}$.}
\end{cases}
\]
\end{corollary}

\section{Functionality}
\label{sec:Fun}
\subsection{Finite Projective Planes}

Recall that a finite projective plane is a pair $(X, \mathcal{L})$, where $X$ is a finite set and $\mathcal{L} \subseteq 2^X$, satisfying the following axioms~\cite{kapitoly_z_dm}:
\begin{enumerate}
  \item
    For every~\(p \neq q \in X\), there is exactly one subset of $X$ in~\(\mathcal{L}\) containing~\(p\) and~\(q\).
  \item
    For every~\(L \neq M \in \mathcal{L}\), we have~\(|L \cap M| = 1\).
  \item
    There exists a subset~\(Y \subseteq X\) of size~4 such that~\(|L \cap Y| \le 2\) for every~\(L \in \mathcal{L}\).
\end{enumerate}
Elements of $X$ are called points and elements of $\mathcal{L}$ are called lines.
We note that for every $k$ which is a power of a prime, a finite projective plane with the following properties can be constructed.
Each line contains exactly $k + 1$ points.
Each point is incident to exactly $k + 1$ lines.
The total number of points is $k^2 + k + 1$.
The total number of lines is also $k^2 + k + 1$.
The number $k$ is called the order of the finite projective plane.

The incidence graph of a finite projective plane is a bipartite graph with one part $X$ and
the second part $\mathcal{L}$. In this graph, $x \in X$ is adjacent to $\ell \in \mathcal{L}$ iff $x$ is incident to $\ell$.
The following theorem shows that the incidence graph of a finite projective plane has functionality approximately $\sqrt{n}$,
where $n$ is the number of vertices.
Alecu et al.~\cite{functionality} have shown that there exists a function~\(f\) such that for every graph~\(G\) we have~\(\vc(G) \le f(\fun(G))\).
Our result complements this inequality by showing that the functionality of a graph cannot be upper bounded by its VC-dimension as it is known that the VC-dimension of a finite projective plane (and subsequently of its incidence graph) of any order is~2~\cite{alon_haussler_welzl}.

\begin{theorem}
\label{thm:Planes}
Consider a finite projective plane of order $k$ and its incidence graph $G$. 
Then, $\fun(G) = k + 1$.
Moreover, for any proper induced subgraph $G' \subset G$ we have $\fun(G') \leq k$.
\end{theorem}

\begin{proof}
First, note that $G$ is a $(k + 1)$-regular graph, thus $\fun(G) \leq k + 1$.
Since $G$ is connected, every proper subgraph $G' \subset G$ contains a vertex of degree at most $k$.
Thus, $\fun(G') \leq k$.

It remains to prove that the functionality of every vertex $v \in V(G)$ is at least $k + 1$.
Let $\ell$ be a line of the projective plane. 
Because of the point-line duality of finite projective planes~\cite{kapitoly_z_dm}, it is enough to show that the functionality of $\ell$ is at least $k + 1$. 
Let $S = \{p_1, \dots, p_a, \ell_1,\dots, \ell_b\}$ be a set of points $p_i$ and lines $\ell_j$ (distinct from $\ell$) such that $|S| = a + b \leq k$. 
We will prove that there exist vertices $u$ and $w$ satisfying:
\begin{enumerate}
 \item The vertices $u$ and $w$ are not in $S$ and they are not adjacent to any vertices in $S$. They are also different from $\ell$.
 \item The vertex $u$ is adjacent to $\ell$, i.e., $u$ is a point incident to $\ell$.
 \item The vertex $w$ is not adjacent to $\ell$, i.e., $w$ is a point not incident to $\ell$ or it is a line.
\end{enumerate}
Once we prove the existence of these two vertices, we are done as their existence implies that $\ell$ is not a function of $S$.

There are $k + 1$ points on each line and every two lines intersect in one point.
It follows that there is a point $q$ on $\ell$ which is distinct from each $p_i$ and it is not incident with any $\ell_i$. 
Thus, the vertex $q$ is not adjacent to any vertex in $S$ and it is adjacent to $\ell$.
We set $u = q$.

Let $L = \{\ell, \ell_1,\dots,\ell_b\}$.
We will prove the existence of the vertex $w$ by considering two cases.
The first case is $b = k$, i.e., the set $S$ consists of $k$ lines and no points.
We may consider any line $\ell' \not \in L$ as the vertex $w$.
The vertex $w$ is not adjacent to any vertex in $S$ as it contains only lines and it is not adjacent to the line $\ell$ either.
Thus, the vertex $w$ satisfies the sought properties.

The second case is $b \leq k - 1$.
Let $P$ be the set of points  $p_1, \ldots, p_a$ together with all the points that are incident to some line from $L$.
We will show that $|P| \leq k^2 + k$, which implies the existence of a vertex $w \in V(G) \setminus P$ satisfying the sought properties.
First, note that if $b = 0$ (i.e., $S$ contains only points), then $|P| \leq k + 1 + a \leq 2k + 1$.
Next, we suppose that $b \geq 1$ which means that $L$ contains at least two lines.
Recall that each pair of lines intersects at a point.
Thus, there are at most $(b+1)(k+1) - 1$ points incident to the lines in $L$ (the $-1$ comes from the fact that we have at least two lines).
Moreover, we have $a$ points $p_1, \dots, p_a$. 
Thus, $|P| \leq (b+1)(k+1) - 1 + a$.
Since $a + b \leq k$, we have $a \leq k - b$.
Therefore,
\[
 |P| \leq (b+1)(k+1) - 1 + k - b = bk + 2k \leq k^2 + k.
\]
This concludes the second case and also the whole proof.
\end{proof}

\subsection{Upper Bound for General Graphs}

\begin{theorem}
\label{thm:GeneralFunUB}
  If $G$ is a graph with $n$ vertices, then $\fun(G) \le \sqrt{c\cdot n \ln n}$ for any $c>3$, 
  if $n$ is big enough. 
\end{theorem} 

\begin{proof} 
 We show that there is $v \in V(G)$ such that $\fun_G(v) \le d(n) =  \sqrt{c \cdot n \ln n}$.
 As $d(n)$ is increasing, this suffices to show the existence of such vertex also in all 
 induced subgraphs of~$G$. We will write $d = d(n)$. 

 Let $D(u,v) = (N(u) \sym N(v)) \setminus \{u,v\}$, i.e., $|D(u,v)| = \sd(u,v)$.
 First, suppose that $\sd(u,v) < d$ for some $u$, $v$ of $G$.
 Then, the vertex~$v$ is a function of the set $D(u,v) \cup \{u\}$, i.e., of at most $d$ vertices.
 Next, suppose all sets $D(u,v)$ have at least $d$ vertices. 
 In this case, we choose an arbitrary vertex~$v \in V(G)$. 
 We also choose a random set $S \subseteq V(G)$ by independently putting each vertex of~$G$ different from $v$ to~$S$ with probability~$p = d/n$. 
 Suppose $v$ is not a function of $S$.
 Then, there exists $u_1 \in N(v) \setminus S$ and $u_2 \notin N(v) \cup S \cup \{v\}$ such that neighbors of $u_1$ and $u_2$ in $S$ are the same, that happens if and only if $D(u_1,u_2) \cap S = \emptyset$.
 We bound the probability of this event by the union bound 
 \begin{align*}
   \Pr\bigl[ \exists u_1 & \in N(v) \setminus S, u_2 \notin N(v) \cup S \cup \{v\} : S \cap D(u_1,u_2) = \emptyset\bigr] \\
   &\leq \sum_{u_1, u_2} (1-p)^{\sd(u_1,u_2)-1} \le n^2 (1-p)^{d-1}. 
 \end{align*}
 
 The $-1$ in the exponent is caused by $v \in D(u_1,u_2)$, but $v$ cannot be chosen to belong to~$S$. 
 Thus, the probability that $v$ is not a function of $S$ is at most $n^2 e^{-p(d-1)}$, which is
 strictly smaller than $\frac{1}{n}$ whenever $c > 3$ and $n$~is big enough. 
 Clearly, the expected size of~$S$ is $p(n-1) = \frac{n-1}{n} d$. 
 Thus by Markov inequality,  $\Pr[|S| > d] \le \frac{n-1}{n} = 1 - \frac{1}{n}$. 
 This means that with a positive probability $|S| \le d$ and $v$ is a function of $S$ and we can conclude that $\fun_G(v) \leq d$.
\end{proof} 

%

\subsection{Random Graphs}
In this section, we prove our bounds for the functionality of random graphs.
Recall that $\fun\bigl(G(n,p)\bigr) = \fun\bigl(G(n,1-p)\bigr)$.
Thus in this section, we always assume that $p \leq \frac{1}{2}$.

\subsubsection{Lower Bound}
We state the lower bound using the following function defined for the non-negative real numbers.
Let $W: \R_0 \to \R_0$ be the inverse function\footnote{This function is usually called the 0-branch of the Lambert function.} of the function $x \mapsto xe^x$.
One can easily verify the following bounds for $W$.
\begin{equation}
\label{eq:WBound}
W(x) =
\begin{cases}
\Theta(\ln x) & \text{if $x > e$} \\
\Theta(x)  & \text{if $0 \leq x \leq e$}
\end{cases}
\end{equation}

\begin{theorem}
\label{thm:RandomLowerBound}
Let $p = p(n) \in \left[\frac{\ln n}{n},\frac{1}{2}\right]$.
The functionality of the graph $G(n,p)$ is with probability at least $1 - o(1)$ larger than
\[
\Omega\left(\frac{1}{p} \cdot W\left(\frac{p^2n}{\ln n}\right)\right).
\]
\end{theorem}
\begin{proof}
Let $k = \left\lfloor \frac{c}{p} \cdot W\left(\frac{p^2n}{\ln n}\right) \right\rfloor$ for an appropriate constant $c$ that will be set later and let $n$ be large enough.
We assume that $k \geq 1$, otherwise there is nothing to prove (note that for $p \geq \frac{\ln n}{n}$ the graph $G(n,p)$ is non-empty with high probability and thus, $\fun\bigl(G(n,p)\bigr) \geq 1$).
Using (\ref{eq:WBound}), it is straightforward to verify $k \leq O(\sqrt{n \cdot \ln n})$ for all $p$.
We will show that the functionality of the graph $G(n,p) = (V, E)$ is at most $k$ with only a small probability.
First, we fix vertices $v, u_1,\dots, u_k \in V$ and we bound the probability that $v$ is a function of $u_1,\dots,u_k$.
Let $B \subseteq V \setminus \{v,u_1,\dots,u_k\}$ be a set of vertices of $G$ that are not connected to any of the vertices $u_1,\dots, u_k$.
By the definition of functionality, if $v$ were a function of $u_1,\dots,u_k$ then either all vertices from $B$ would be connected to $v$ or none of them.
Thus, we have
\[
\Pr[v \text{ is a function of } u_1,\dots,u_k] \leq p^{|B|} + (1-p)^{|B|} \leq 2\cdot (1-p)^{|B|}.
\]

Each vertex different from $v,u_1,\dots,u_k$ is in the set $B$ with probability ${(1 - p)^k}$, independently of the other vertices.
Thus, $\mu = \E[|B|] = (1 - p)^k\cdot (n - k - 1)$ and we bound $|B|$ by Chernoff bound~\cite{prob_comp_book}:
\[
\Pr\bigl[ |B| \leq (1 - \delta)\mu \bigr] \leq e^{\frac{-\delta^2 \mu}{2}}.
\]
Let $b = (1 - \delta)\mu$.
Putting it all together, we have
\begin{align*}
\Pr[& v \text{ is a function of } u_1,\dots,u_k] \leq e^{\frac{-\delta^2 \mu}{2}} + 2\cdot (1-p)^{b} \\
 &\leq e^{\frac{-\delta^2 \mu}{2}} + 2\cdot e^{-p(1 - \delta)\mu} \\
 &\leq 3 \cdot e^{\frac{-p}{4} \mu} \tag{set $\delta^2 = p$, and $p \leq \frac{1}{2}$ by the assumption}
\end{align*}
There is $n\cdot \binom{n-1}{k} \leq n^{k+1}$ choices how to pick $v,u_1,\dots,u_k$.
Therefore, by union bound, we have the following.
\begin{align*}
\Pr\bigl[\fun\bigl(G(n,p)\bigr) &\leq k\bigr] \leq n^{k + 1} \cdot 3 e^{\frac{-p}{4}\mu} \\
&= \exp\bigl((k+1)\cdot \ln n - \frac{p}{4}\cdot(1-p)^k\cdot (n - k - 1) + \ln 3\bigr) \\
&\leq \exp\bigl(2k\cdot\ln n - \frac{p}{5}\cdot (1-p)^k n\bigr) \tag{by $1 \leq k \leq o(n)$}
\end{align*}
Using our choice of $k$, we will show that
\begin{equation}
\label{eq:Exponent}
\frac{p}{5}\cdot (1-p)^k n \geq 3 k \cdot \ln n
\end{equation}
which implies that $\Pr\bigl[\fun(G(n,p)) \leq k\bigr] \leq 1/n^k$.

Choose $\gamma > 0$ such that $e^{-\gamma} = 1-p$.
By rearranging (\ref{eq:Exponent}) (and multiplying both sides by $\gamma$)  we get the following.
\[
\frac{\gamma p n}{15 \ln n} \geq \frac{\gamma k}{(1-p)^k} = \gamma k \cdot e^{\gamma k}
\]
After applying the function $W$ we get the following inequality.
\begin{equation}
\label{eq:WBoundForK}
\gamma k \leq W\left(\frac{\gamma p n}{15 \ln n}\right) 
\end{equation}

Since we assume that $p \leq \frac{1}{2}$, we have that $e^{-2p} \leq 1-p \leq e^{-p}$ and thus $p \leq \gamma \leq 2p$.
Plugging it into (\ref{eq:WBoundForK}), we get the following.
\begin{align*}
 k &\leq \frac{1}{\gamma} \cdot W\left(\frac{\gamma p n}{15 \ln n}\right) \leq \frac{1}{p} \cdot W\left(\frac{2 p^2 n}{15 \ln n}\right) \tag{since $W$ is increasing} \\
 &= \frac{1}{p} \cdot \Theta\left(W\left(\frac{p^2 n}{\ln n}\right)\right) \tag{by (\ref{eq:WBound})}
\end{align*}
Thus, for a right constant $c$, our setting of $k$ satisfies (\ref{eq:WBoundForK}) and since $W$ is increasing, it also satisfies ($\ref{eq:Exponent}$).
Consequently, it yields the sought lower bound.

\end{proof}

Using (\ref{eq:WBound}), Theorem~\ref{thm:RandomLowerBoundFinal} is a corollary of Theorem~\ref{thm:RandomLowerBound}.
\RandomLowerBoundFinal*
\begin{proof}
 By Theorem~\ref{thm:RandomLowerBound}, we have that $\fun \bigl(G(n,p)\bigr) \geq \Omega\left(\frac{1}{p} \cdot W\left(\frac{p^2 n}{\ln n}\right)\right)$ w.h.p.
 If $p \geq \frac{1}{n^{\frac{1}{2}-\varepsilon}}$, then $\frac{p^2n}{\ln n} \geq \frac{n^{2\varepsilon}}{\ln n}$ and thus, by the first bound of (\ref{eq:WBound}), we have that
 \[
  \frac{1}{p} \cdot W\left(\frac{p^2 n}{\ln n}\right) \geq \frac{1}{p} \cdot \Omega\bigl(\ln n^{2\varepsilon}\bigr) = \Omega\left(\frac{\varepsilon}{p} \cdot \ln n\right). 
 \]
 
 If $\frac{1}{\sqrt{n}} \leq p \leq \frac{1}{n^{\frac{1}{2} - \varepsilon}}$, then by (\ref{eq:WBound}), we have $W\left(\frac{p^2 n}{\ln n}\right) \geq \Omega\left(\frac{1}{\ln n}\right)$ in both cases when $\frac{p^2 n}{\ln n} > e$ and $\frac{p^2 n}{\ln n} \leq e$. 
 
 If $\frac{\ln n}{n} \leq p \leq \frac{1}{\sqrt{n}}$, then $\frac{p^2 n}{\ln n} < 1$ and thus by the second bound of (\ref{eq:WBound}), we have that
 \[
  \frac{1}{p} \cdot W\left(\frac{p^2 n}{\ln n}\right) \geq \frac{1}{p} \cdot \Omega\left(\frac{p^2 n}{\ln n}\right) = \Omega\left(\frac{pn}{\ln n}\right).
 \]

\end{proof}

\subsubsection{Upper Bound}
In this section, we will show the upper bound.
One of the tools used for the upper bound is the approximation of balls and bins distribution by independent Poisson random variables.
Suppose we have $m$ balls and $n$ bins.
Each ball is thrown into the $i$-th bin with probability $p_i$, independently of the other throws (and $\sum_{i \in [n]} p_i = 1$, i.e., each ball is thrown into some bin).
Let $X_i$ be the number of balls in the $i$-th bin after all the balls are thrown.
Let $Y_i$ be a random variable of Poisson distribution such that $\E[Y_i] = p_i m$ and is independent of other variables $Y_j$'s.
Note that $\E[X_i] = p_i m$ as well, however the variables $X_1,\dots,X_n$ are not independent.
Nevertheless, the variables $Y_i$'s yield a good approximation for the variables $X_i$'s for events of small probability, as stated in the following theorem.

\begin{theorem}[\cite{prob_comp_book}]
\label{thm:PoisApx}
Let $A \subseteq \N^n$ be an event.
Then,
\[
\Pr[(X_1,\dots,X_n) \in A] \leq e^{1/12m}\sqrt{2 \pi m}\cdot \Pr[(Y_1,\dots,Y_n) \in A]
\]
\end{theorem}
Mitzenmacher and Upfal~\cite{prob_comp_book} only gave a proof of the previous theorem for the uniform setting, where each $p_i = \frac{1}{n}$, i.e., each bin has the same probability, that we throw a ball into this bin.
However, the proof for the non-uniform setting is the same and we include it in the appendix for the sake of completeness.
We remark that in Theorem~\ref{thm:PoisApx}, the factor\footnote{Mitzenmacher and Upfal stated this theorem with the constant $e$ instead of $e^{1/12m}\sqrt{2\pi}$ because of usage of a looser bound for $m!$.} $e^{1/12m}\sqrt{2 \pi m}$ can be replaced by 2 for monotone events (the probability of the event is either increasing or decreasing when the number of balls $m$ is increasing).
That will be also our case, however, we use this simpler variant as this improvement would not asymptotically improve our result.
Further, we will use a more precise variant of well-known bounds for the binomial coefficient (see for example the book by Alon and Spencer~\cite{prob_method_book}).
We state a proof in the appendix.

\begin{restatable}{lemma}{BinomialBound}
\label{lem:BinomialBound}
Let $\alpha \in (0,\frac{1}{2})$ and $n \geq 2$. Then,
\[
 \frac{1}{22\cdot n^{3/2}}\cdot  e^{H(\alpha ) n} \leq \binom{n}{\lfloor \alpha n \rfloor} \leq e^{H(\alpha) n},
\]
where $H(\alpha) = -\alpha \ln \alpha - (1-\alpha) \ln (1 - \alpha)$.
\end{restatable}

As discussed in Section~\ref{sec:Intro}, the upper bound for the functionality of random graphs follows from Theorem~\ref{thm:RandomUpperBound}.
Recall that a distinguishing set of vertices is a set $D \subseteq V(G)$ such that there are no two vertices $u,v \in V(G) \setminus D$ with $N(u) \cap D = N(v) \cap D$.
\RandomUpperBound*

\begin{remark}
 The proof of Theorem~\ref{thm:RandomUpperBound} works even for $p \geq 1/n^d$ where $d$ is an arbitrary constant. However, for $p$ smaller than $1/\sqrt{n}$ the theorem would provide an upper bound $O(\sqrt{n} \cdot \ln n)$ which is higher than the upper bound for all graphs given by Theorem~\ref{thm:GeneralFunUB}.
\end{remark}

In the following proof we use notation $[z] = \{1,\dots,z\}$ for $z \in \N$, and for a set $S$ and $k \in \N$, $\binom{S}{k}$ denotes the set of all subsets of $S$ of size $k$.
\begin{proof}
Let $k = \frac{c'}{p} \ln n$ for a sufficiently large constant $c'$ and let $n$ be large enough.
We will show that each induced subgraph of $G=G(n,p)$ has a distinguishing set of size at most $6k$.
For a set $S \subseteq V(G)$ of size~$s$ we let $r_s$ be the probability that the induced subgraph $G[S]$ does not have a distinguishing set of size at most $6k$.
(Clearly, this is the same for all sets of the same size.)
Note that if $|S| \leq 6k$, then $G[S]$ always has a distinguishing set of size at most $6k$.
Next, we provide an estimate for $r_s$ that will allow us to show
\[
  \sum_{s \ge 6k} \binom{n}{s} r_s = o(1),
\]
which will finish the proof.
In particular, we will show that for all $s \ge 6k$ holds that $r_s \leq n^{-2s}$ (if $c'$ is large enough).

Fix a set $S \subseteq V(G)$ of size $s \geq 6k$ and write $S = S_1 \dot\cup S_2 \dot\cup S_3$ with each $|S_i|$ being $\lfloor s/3\rfloor$ or $\lceil s/3\rceil$.
For $i = 1, 2$ or~$3$ we let $T_i = S \setminus S_i$.
We will try to use a set $D \in \binom{S_i}{k}$ to distinguish vertices in~$T_i$.
We will sometimes fail, but barely.
Formally, we say a set $D \in \binom{S_i}{k}$ is \emph{almost distinguishing} for $T_i$, if there is a set $E_D \subseteq T_i$ of exceptional vertices such that $|E_D| \leq k$ and for each pair of different vertices $u,v \in T_i \setminus E_D$ we have $N(u) \cap D \neq N(v) \cap D$, i.e., $u$ and $v$ have a different neighborhood in $D$.
We let $A_i$ be the event that no set~$D \in \binom{S_i}{k}$ is almost distinguishing for $T_i$.
Suppose none of~$A_i$ occurs.
Then, each $S_i$ contains a set $D_i$ that is almost distinguishing for $T_i$.
Then $\Dbar = \cup_{j=1}^3 (D_j \cup E_{D_j})$ is a distinguishing set in $G[S]$ of size at most~$6k$.
Indeed, let $u,v$ be vertices in $S \setminus \Dbar$.
Then, there is $i \in \{1,2,3\}$ such that $u,v \not \in S_i$, i.e., $u,v \in T_i$.
The set $D_i \subseteq S_i$ distinguishes all vertices in $T_i \setminus E_{D_i}$.
Since $u,v$ are not in $E_{D_i}$ as well, they have distinct neighborhoods in $D_i \subseteq D$, i.e., $N(u) \cap D_i \neq N(v) \cap D_i$.
Thus, the set $\Dbar$ is a distinguishing set of $G[S]$ and we have that $r_s \le \Pr[A_1] + \Pr[A_2] + \Pr[A_3]$.
See Figure~\ref{fig:DistinguishingSet} for an illustration of how we create the distinguishing set $\Dbar$.

\begin{figure}[ht]
 \centering
 \includegraphics{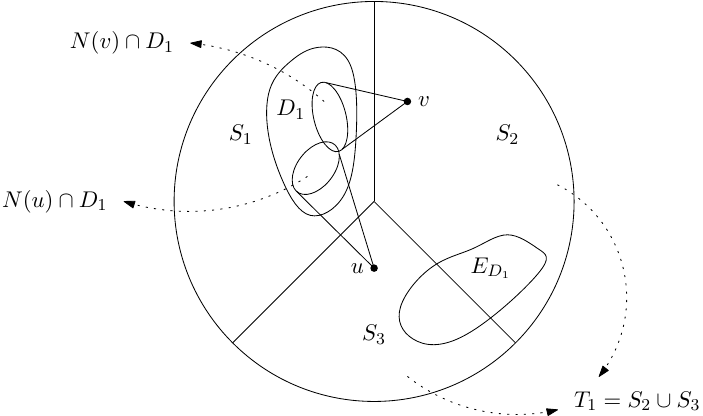}
 \caption{We divide the vertices of $G[S]$ into three sets $S_1,S_2$, and $S_3$ with roughly the same size.
 Each $S_i$ contains an almost distinguishing set $D_i$ for $T_i$ with the set $E_{D_i}$ of the exceptional vertices.
 Then, each pair of vertices not in $\Dbar = \bigcup_{j = 1}^3 D_j \cup E_{D_j}$ is at least in one set $T_i \setminus E_{D_i}$ and thus
they can be distinguished by the neighborhood in $D_i \subseteq \Dbar$, i.e., $\Dbar$ is a distinguishing set of $G[S]$.}
 \label{fig:DistinguishingSet}
\end{figure}

We will now estimate $\Pr[A_1]$ ($= \Pr[A_2] = \Pr[A_3]$).
For a particular set $D \in \binom{S_1}{k}$ let us study when $D$ is not almost distinguishing for $T_1$:
each vertex of~$T_1$ falls in one of~$2^k$ ``bins'' based on the neighborhood with~$D$.
Formally, for $L \subseteq D$, let $B_L = \{v \in T_1 \mid N(v) \cap D = L\}$.
We choose arbitrarily one vertex from each (nonempty) bin $B_L$ and put all the rest to the set~$E_D$ of
exceptional vertices.
Thus,
\[
  |E_D| = \sum_{L \subseteq D} \max\{0, |B_L| - 1\}.
\]
We want $|E_D| \le k$.
Let $q = \Pr[|E_D| > k]$.
A crucial observation is that if we take \emph{disjoint} sets~$D, D' \in \binom{S_1}{k}$ then the events $|E_D| > k$ and $|E_{D'}| > k$ are independent.
Let $D_1,\dots,D_z \in \binom{S_1}{k}$ be disjoint subsets of $S_1$ for $z = \lfloor |S_1| / k \rfloor$.
We bound the probability of the event $A_1$ by the probability that none of the sets $D_1,\dots,D_z$ is almost distinguishing for $T_1$, i.e.,
\begin{equation}
   \Pr[A_1] \le \Pr\bigl[\forall j \in [z]: |E_{D_j}| > k\bigr] \leq  q^{\bigl\lfloor \frac{\lfloor s/3 \rfloor}{k} \bigr\rfloor} \leq q^{s/ 12 k},
   \label{neq:ABound}
\end{equation}
where the last inequality holds because $s \geq 6k$.

Now let us find, or rather estimate $\Pr[|E_D| > k]$ for a set $D \in \binom{S_1}{k}$.
To estimate this random variable we use the Poisson approximation for balls-into-bins given by Theorem~\ref{thm:PoisApx}.
Note that $(|B_L|)_{L \subseteq D}$ has exactly the same distribution as if we throw $t = |T_1|$ balls into $2^k$ bins, where the probability of throwing a ball (vertex) into the bin $B_L$ is $p_\ell = p^\ell(1-p)^{k-\ell}$ for $\ell = |L|$.
Thus, we replace $|B_L|$ by $Y_L$, where $Y_L$ has Poisson distribution with $\E[Y_L] = p_{|L|} \cdot t$ and the random variables $(Y_L)_{L \subseteq D}$ are independent.
Let $Z_L = \max\{0, Y_L-1\}$ and $Z = \sum_{L \subseteq D} Z_L$.
By Theorem~\ref{thm:PoisApx}, we have $\Pr[|E_D| > k] \leq 3\sqrt{t}\cdot \Pr[Z > k]$.

To estimate this further, we use the standard ``Chernoff trick'':
$\Pr[Z > k] = \Pr[ e^{hZ} > e^{h k} ] < \E[e^{h Z}] e^{-h k}$ for any $h > 0$.
As $Z$ is a sum of independent random variables, we have $\E[e^{h Z}] = \prod_{L \subseteq D} \E[e^{h Z_L}]$.
Let $L \subseteq D$ and $\ell = |L|$.
We denote $\lambda_\ell = \E[Y_L] = p_\ell \cdot t$.
Let us bound $E[e^{h Z_L}]$, writing $w = e^h $.
\begin{align}
    \E[e^{hZ_L}] &= e^{0} \Pr[Y_L = 0] + \sum_{i=1}^\infty e^{h (i-1)} \Pr[Y_L=i] \nonumber \\
                 &= e^{-\lambda_\ell} + \sum_{i=1}^\infty e^{h (i-1)} e^{-\lambda_\ell} \frac{\lambda_\ell^i}{i!} \nonumber \\
                 &= e^{-\lambda_\ell} + e^{-\lambda_\ell - h } \sum_{i=1}^\infty \frac{(\lambda_\ell e^h )^i}{i!} \nonumber \\
                 &= e^{-\lambda_\ell} + e^{-\lambda_\ell - h } ( e^{\lambda_\ell e^h }  - 1)  \nonumber \\
                 &= e^{-\lambda_\ell}( 1 + ( e^{\lambda_\ell w}  - 1)/w ) \nonumber  \\
                 &\le e^{-\lambda_\ell}( 1 + \lambda_\ell + \lambda_\ell^2 w)  
                 \tag{as $e^x \le 1+x+x^2$ on $[0,1]$}  \\
                 \label{neq:ZL_bound}
  \end{align}

  Next, we argue that the function $f(\lambda) = e^{-\lambda}(1 + \lambda + \lambda^2w)$ is increasing in $\lambda$ in our settings of parameters.
  Note that $f'(\lambda) = e^{-\lambda} \lambda \bigl(w(2 - \lambda) - 1\bigr)$.
  We have $w > 1$ and thus, $f'(\lambda) > 0$ for $0 \leq \lambda \leq 1$.
  If $c'$ is large enough, then for all $\ell$ holds that
  \[
   \lambda_\ell = p^\ell (1-p)^{k-\ell} \cdot t \leq (1-p)^k \cdot n \leq e^{-c' \ln n}  \cdot n \leq 1.
  \]
  Thus, $f(\lambda_\ell) \leq f(\lambda_{\ell'})$ if $\lambda_\ell \leq \lambda_{\ell'}$ (which holds when $\ell \geq \ell'$) and we can upper bound $\E[e^{hZ_L}]$ by $e^{-\lambda_0} (1 + \lambda_0 + \lambda_0^2 w)$ (using (\ref{neq:ZL_bound})).
  We distinguish two cases based on whether $p \geq 0.3$ or not.

\paragraph*{Case 1: $0.3 \leq p \leq \frac{1}{2}$.}
Since $\lambda_0 \geq \lambda_\ell$ for all $\ell \in \{0,\dots,k\}$ as $p \leq \frac{1}{2}$, we have that for all $Z_L$ holds that
$\E[e^{hZ_L}] \leq f(\lambda_{|L|}) \leq f(\lambda_0) = e^{-\lambda_0}( 1 + \lambda_0 + \lambda_0^2 w)$, by (\ref{neq:ZL_bound}) and $f$
being increasing on $[0,1]$. Further, $1 + \lambda_0 + \lambda_0^2 w \le (1+\lambda_0)(1+\lambda_0^2 w)$.
Consequently we have,
\begin{equation}
    \Pr[Z > k] \le w^{-k} \bigl( e^{-\lambda_0}( 1 + \lambda_0 + \lambda_0^2 w)  \bigr)^{2^k} \le w^{-k} \cdot \exp ( \lambda_0^2 \cdot w \cdot 2^k) \label{neq:ProbZbound}
\end{equation}

We set $w = \frac{1}{\lambda_0^2 \cdot 2^k}$.
We need to verify that $w > 1$ (as $w = e^h$ for some $h > 0$) and $\lambda_0\cdot w < 1$ so that (\ref{neq:ZL_bound}) holds.
\[
\lambda_0 \cdot w = \frac{1}{\lambda_0 \cdot 2^k} = \frac{1}{(1-p)^{k}\cdot t \cdot 2^k} \leq \frac{1}{t}
\]
The last inequality holds since $p \leq \frac{1}{2}$.
Further, we will actually show that $w > n^d$ for arbitrary large constant $d$ if we set $c'$ large enough.
\begin{align}
w &= \frac{1}{\lambda_0^2 \cdot 2^k} = \frac{1}{(1-p)^{2k} \cdot t^2 \cdot 2^k} \nonumber \\
&\geq \frac{1}{n^2 \cdot \bigl(2(1-p)^2\bigr)^k} \tag{since $t \leq n$} \nonumber \\
&\geq \frac{1}{n^2\cdot 0.98^k} \tag{since $p \geq 0.3$} \nonumber \\
&\geq n^d \tag{for $c'$ large enough as $k \geq 2c' \ln n$} \\
\label{neq:WBound}
\end{align}

Thus, it remains to plug $w$ into (\ref{neq:ProbZbound}) to get the final bound for $\Pr[|E_D| > k]$.
  \begin{align*}
    q &= \Pr[|E_D| > k]  \leq 3\sqrt{t}\cdot \Pr[Z > k] \tag{by Theorem~\ref{thm:PoisApx}} \\
    &\leq 3\sqrt{n} \cdot w^{-k} \cdot e^1  \tag{by (\ref{neq:ProbZbound})} \\
    &\leq n^{-dk/2} \tag{by (\ref{neq:WBound})}
  \end{align*}

Now, we finish the bound that $r_s < n^{-2s}$ (recall that $r_s$ is the probability that an induced subgraph $G'$ on $s$ vertices does not contain a distinguishing set of size $6k$).
\begin{align*}
  r_s &\le \Pr[A_1] + \Pr[A_2] + \Pr[A_3] \leq 3 q^{s/12k} \tag{by (\ref{neq:ABound})}\\
  &\le 3 n^{-d s / 24} \leq n^{-2 s}  \tag{if $d$ is large enough}
\end{align*}
Thus, we finished the proof of the case when $p \geq 0.3$.

\paragraph*{Case 2: $\frac{1}{\sqrt{n}} \leq p < 0.3$.}
Proof of this case is similar to the previous one.
However, we will consider two types of bins.
Note that for a vertex in $T_1$ and a set $D \in \binom{S_1}{k}$, the expected number of neighbors in $D$ is $pk$.
Thus, for bins $B_L$ where $|L|$ is around $pk$ we use the previous calculation to show that they contain at most $k/2$ exceptional vertices with high probability.
Further, we will show that all other bins will contain with high probability at most $k/2$ vertices in total.
Thus again, it will hold that $|E_D| \leq k$ with high probability.

First, we bound the number of vertices in the bins with the non-typical size of the neighborhood in $D$.
For a vertex $v \in T_1$, let $N_v$ be the number of its neighbors in $D$.
Note that $N_v = \sum_{u \in D} I^v_u$, where $I^v_u$ is an indicator random variable such that $I^v_u = 1$ when there is an edge $\{v,u\}$.
We have $\mu = \E[N_v] = pk =  c'\cdot \ln n$ and by Chernoff bound for any $\delta \in (0,1)$:
\[
\Pr\bigl[|N_v - \mu| \geq \delta \mu\bigr] \leq 2\cdot \exp\left(-\frac{\delta^2}{3} \cdot c' \cdot \ln n\right).
\]
Note that the random variables $N_v$'s for $v \in T_1$ are independent.
We set $\delta = \frac{1}{10}$ and we conclude the following bound:
\begin{align}
\Pr&\left[\exists ~ \frac{k}{2} \text{ vertices } v \in T_1: |N_v - \mu| \geq \frac{\mu}{10}\right] 
   \leq \binom{t}{k/2} \cdot 2^{k/2} \cdot \exp\left(-\frac{k}{600} \cdot c' \ln n\right) \nonumber \\
   &\leq \exp\left(\frac{k}{2} \ln n  -\frac{k}{600} \cdot c' \ln n + \frac{k}{2} \ln 2 \right) \leq n^{-dk}, \label{neq:NvBound}
\end{align}
where $d$ can be again arbitrarily large constant if we set $c'$ large enough.
Let $I = \{ \lceil 0.9 \cdot c' \ln n \rceil, \dots, \lfloor 1.1 \cdot c' \ln n \rfloor \}$ and let $\cI$ be the set of subsets of $D$ with size in $I$, i.e., $\cI = \{L \subseteq D : |L| \in I\}$.
By (\ref{neq:NvBound}), it holds that with high probability there are at most $\frac{k}{2}$ vertices in all bins $B_L$ with $L \not \in \cI$.

For the bins $B_L$ with $L \in \cI$, we use a similar calculation as in the previous case.
Recall that $Z_L = \max\{0,Y_L - 1\}$, where $Y_L$ is a random variable with Poisson distribution with $\E[Y_L] = \lambda_{|L|} = p_{|L|} \cdot t = p^{|L|}(1-p)^{k - |L|}\cdot t$.
Let $Z' = \sum_{L \in \cI} Z_L$ and subsequently, $\Pr[Z' > k/2] = \Pr[e^{hZ'} > e^{hk/2}] < \E[e^{hZ'}]e^{-hk/2}$ for $h > 0$.
By (\ref{neq:ZL_bound}), for $|L| = \ell$ and $w = e^h$ we have $\E[e^{hZ_L}] \leq e^{-\lambda_\ell}(1 + \lambda_\ell + \lambda^2_\ell w)$, if $\lambda_\ell w < 1$.
Let $\lambda' = \lambda_{\ell'}$ for $\ell' = \lceil 0.9 \cdot c' \ln n \rceil$ (i.e., $\ell'$ is the minimum of $I$, and $\lambda'$ is the maximum of $\lambda_\ell$ for $\ell \in I$).
For $a = |\cI|$, we have
\begin{align}
\Pr[Z' > k/2] &< w^{-k/2}\cdot \E[e^{hZ'}] =  w^{-k/2}\cdot \prod_{L \in \cI} \E[e^{hZ_L}] \nonumber \\
&\leq w^{-k/2} \cdot \bigl(e^{-\lambda'}( 1 + \lambda' + \lambda'^2 w)\bigr)^a \tag{by (\ref{neq:ZL_bound}) and definition of $\lambda'$} \nonumber \\
&\leq w^{-k/2} \cdot \exp(\lambda'^2 \cdot w \cdot a) \label{neq:ZprimeBound}
\end{align}
We set $w = \frac{1}{\lambda'^2 a}$.
Again, we will show that $\lambda' w < 1$ (so that the bound (\ref{neq:ZL_bound}) holds) and $w \geq n^d$ for arbitrarily large constant $d$ if $c'$ is large enough.
To do so, we will use two technical inequalities that are stated (as Propositions~\ref{prp:TechnicalForLambdaW} and~\ref{prp:TechnicalForW}) and proved in the appendix.
Next, we estimate $a$.
\[
\binom{k}{\lfloor 1.1 \cdot c' \ln n \rfloor} \leq a \leq |I|\cdot \binom{k}{\lfloor 1.1 \cdot c' \ln n \rfloor}
\]
Note that $|I| = \lfloor 1.1 \cdot c' \ln n \rfloor - \lceil 0.9 \cdot c'  \ln n \rceil \leq 0.2 \cdot c'  \ln n$ and $1.1 \cdot c' \ln n = 1.1 p \cdot k$.
Thus, by Lemma~\ref{lem:BinomialBound} we have
\begin{equation}
\frac{1}{22 \cdot k^{3/2}} \cdot e^{H(1.1p) k} \leq a \leq 0.2 \cdot c' \ln n \cdot e^{H(1.1p) k}. \label{neq:aBound}
\end{equation}
Now, we show that $\lambda' w < 1$.
\begin{align}
\lambda' w &= \frac{1}{p^{\ell'} (1-p)^{k - \ell'} \cdot t \cdot a} \nonumber \\
&\leq \frac{22 \cdot k^{3/2}}{p^{\ell'} \cdot e^{-\frac{5p}{4}(k-\ell')}\cdot   e^{H(1.1p) k}} \tag{by (\ref{neq:aBound}) and $(1-p) \geq e^{-\frac{5p}{4}}$ for $p \leq 0.3$} \nonumber \\
&= 22 \cdot k^{3/2} \cdot \exp \left( \frac{5p}{4}(k-\ell') -\ell'\ln p - H(1.1p) k \right) \label{neq:LambdaWBound}
\end{align}
We expand the exponent of (\ref{neq:LambdaWBound}) with $\ell' = \lceil 0.9 \cdot c' \ln n \rceil$ and $k = \frac{c'}{p} \ln n$.
\begin{align*}
\frac{5p}{4}&(k-\ell') -\ell'\ln p - H(1.1p)k \leq c'\ln n \cdot \left(\frac{5}{4} - \frac{45p}{40} - \frac{9}{10} \ln p - \frac{H(1.1p)}{p} \right) \\
&\leq -\frac{c'}{10} \ln n \tag{for $p \leq 0.3$ by Proposition~\ref{prp:TechnicalForLambdaW} in the appendix}
\end{align*}
We plug this bound into (\ref{neq:LambdaWBound}) to get the following bound, $\lambda' w \leq 22\cdot k^{3/2} n^{-c'/10}$ which is smaller than 1 if $c'$ is large enough as $k \leq c'\sqrt{n}\ln n$.

Analogously, we will show that $w \geq n^d$ for a large enough constant $d$.
\begin{align*}
w &= \frac{1}{p^{2\ell'}(1-p)^{2(k-\ell')} \cdot t^2 \cdot a}   \\
&\geq \frac{5\cdot \exp\bigl(2(k-\ell')p - 2\ell'\ln p - H(1.1p)k\bigr)}{n^2 \cdot c'\ln n} \tag{by (\ref{neq:aBound})}
\end{align*}
Again, we expand the exponent.
\begin{align*}
2(k &-\ell')p - 2\ell'\ln p - H(1.1p)k \\
&\geq c'\ln n \cdot \left(2 - 2\left(0.9 + \frac{1}{c'\ln n}\right)(p + \ln p) - \frac{H(1.1p)}{p}\right) \tag{as $\ell' \leq 0.9\cdot c' \ln n + 1$} \\
&\geq c'\ln n \cdot \left(2 - 2(p + \ln p) - \frac{H(1.1p)}{p}\right) \tag{if $c'$ is large enough} \\
&\geq c' \ln n \tag{for $p \leq 0.3$ by Proposition~\ref{prp:TechnicalForW} in the appendix}
\end{align*}
Thus for any $d > 0$, we have
\begin{equation}
w \geq \frac{5 n^{c'}}{n^2 \cdot c' \ln n} \geq n^d, \label{neq:WBound2}
\end{equation}
if $c'$ is large enough.
Now, we can finish our bound for $\Pr[Z' > k/2]$.
By definition of $w$, (\ref{neq:ZprimeBound}), and (\ref{neq:WBound2}) we have
\begin{equation}
\Pr[Z' > k/2] \leq w^{-k/2}\cdot \exp (\lambda'^2 \cdot w \cdot a) \leq e\cdot n^{-dk/2}. \label{neq:ProbZprime}
\end{equation}
Further, we bound $q$.
Let $F$ be the event that there are more than $\frac{k}{2}$ vertices in bins $B_L$ with $L \not \in \cI$, and $E'_D$ be the number of exceptional vertices in bins $B_L$ with $L \in \cI$, i.e.,
$E'_D = \sum_{L \in \cI} \max\{0, |B_L| - 1\}$.
\begin{align*}
q &= \Pr\bigl[|E_D| > k\bigr] \leq \Pr[F] + \Pr[\lnot F]\cdot \Pr\bigl[|E_D| > k ~|~ \lnot F\bigr] \\
&\leq n^{-dk} + 2\cdot \Pr\bigl[|E'_D| > k/2\bigr] \tag{by (\ref{neq:NvBound})} \\
&\leq n^{-dk} + 6\sqrt{n} \cdot \Pr[Z' > k/2] \tag{by Theorem~\ref{thm:PoisApx}} \\
&\leq n^{-dk/3} \tag{by (\ref{neq:ProbZprime})}
\end{align*}
To finish the proof of this case (and of the whole theorem) we will conclude again that the probability $r_s$ (that an induced subgraph on $s$ vertices does not contain a distinguishing set of size $6k$) is at most $n^{-2s}$.
\[
r_s \leq 3q^{s/12k} \leq 3n^{-ds/36} \leq n^{-2s} \tag{if $d$ is large enough}
\]
\end{proof}

\section{Symmetric Difference}
\label{sec:SD}
In this section, we will prove our lower and upper bounds for the symmetric difference of interval and circular arc graphs.

\subsection{Circular Arc Graphs}
In this section, we prove that the symmetric difference of circular arc graphs is $\Theta(\sqrt{n})$.
More formally, we prove the following two theorems.

\begin{theorem}
\label{thm:CircArcUB}
 Any circular arc graph $G \in \CA_n$ has a symmetric difference at most $O(\sqrt{n})$.
\end{theorem}
\begin{theorem}
\label{thm:CircArcLB}
 There is a circular arc graph $G \in \CA_n$ of symmetric difference at least $\Omega(\sqrt{n})$.
\end{theorem}
First, we prove the upper bound, i.e., that every circular arc graph has a symmetric difference at most $O(\sqrt{n})$.
We use the following notations for arcs.
Let $a, b$ be two points of a circle.
Then, an arc $r = [a, b]$ is an arc beginning in $a$ and going in a clockwise direction to $b$.
We call $a$ as the starting point of $r$ and $b$ as the ending point of $a$. 

\begin{proof}[Proof of Theorem~\ref{thm:CircArcUB}]
 Let circular arc graph $G = (V,E)$ be an intersection graph of a set of arcs $R = \{r_1,\dots,r_n\}$ of a circle $C$.
 Without loss of generality, we can suppose that the circumference of $C$ is $2n$, and endpoints of all arcs $r_i \in R$ are integer points and are different for all arcs.
 Thus, each integer point $\{0,\dots, 2n-1\}$ of $C$ is an endpoint of exactly one arc in $R$.
 
 Consider an arc $r = [a,b] \in R$.
 We represent the arc $r$ as a point $(a,b)$ in the plane $\R^2$.
 Note that all these points are in the square $S$ with corners in the points $(0,0)$ and $(2n-1,2n-1)$ (some points may be on the border of $S$).
 We divide the square $S$ into subsquares of size $k \times k$ for $k = \frac{2n-1}{\lfloor\sqrt{n}\rfloor - 1}$. 
 Note that we have strictly less than $n$ such subsquares and $k = \Theta(\sqrt{n})$.
 Thus, there is at least one subsquare that contains two points representing arcs, say $r = [a, b]$ and $r' = [a',b']$.
 It follows that $|a - a'|, |b - b'| \leq k$.
 Suppose that $a' > a$ and $b' > b$, other cases are analogous.
 Then, each arc counted in $\sd(r, r')$ has to start or end in an integer point from the interval $[a,a'-1]$ or $[b, b'-1]$.
 Since there are at most $2k$ integer points in these two intervals and each integer point of $C$ is an endpoint of exactly one arc of $R$, we conclude that $\sd(r, r') \leq 2k \leq O(\sqrt{n})$.
\end{proof}

 Now, we give a construction of a circular arc graph of symmetric difference at least $\Omega(\sqrt{n})$.
 Let $n$ be a square of an integer, i.e., $n = d^2$ for some $d \in \N$.
 We consider a circle $C$ of circumference $n$ and a set $P$ of integer points of $C$, i.e., $P = \{0,\dots,n-1\}$ ordered in clockwise direction.
 The length $|r|$ of an arc $r = [a, b]$ is equal to $b - a \pmod n$.
 We say the arc $[a,b]$ is \emph{integral} if both $a$ and $b$ are integers.

 We will represent each point $p \in P$ as two integer indices $0 \leq i,j < d = \sqrt{n}$ such that $p = i\cdot d + j$.
 Note that each point $p$ has a unique such representation and we denote it as $(i,j)_d$.
 Let $R$ be a set of arcs $\bigl[(i,j)_d, (j,i)_d\bigr]$ for all possible $i \neq j$  such that length of each arc in $R$ is at most $\frac{n}{4}$.
 Note that we require that $i \neq j$, thus we do not consider zero-length arcs consisting only of a point of a form $(i, i)_d$.
 See Figure~\ref{fig:circ-arc-lb-construction} for an illustration.
 Let $G$ be the intersection graph of arcs in $R$.
 We we will prove that $\sd(G) \geq \Omega(\sqrt{n})$.
 First, we prove two auxiliary lemmas, Lemma~\ref{lem:ArcsEndingPoints} and~\ref{lem:LongSubarc}.
 Lemma~\ref{lem:ArcsEndingPoints} asserts that each sufficiently long arc of $C$ contains a lot of starting and ending points of arcs in $R$. 
 Lemma~\ref{lem:LongSubarc} states that for any two arcs $r$ and $r'$, we will find a long arc $s$ that is a subarc of only one of the arcs $r$ and $r'$ (say $r$).
 Thus by Lemma~\ref{lem:ArcsEndingPoints}, the arc $r$ intersects many arcs that go ``away'' from the arc $r'$ and that is enough to imply Theorem~\ref{thm:CircArcLB}.

 \begin{figure}
   \centering
   \includegraphics{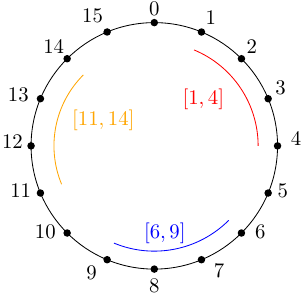}
   \caption{%
     An example of the circular arc graph lower bound construction for~\({n = 16}\) and $d=4$.
     This graph contains three arcs corresponding to points~\({1 = (0, 1)_d}\), \({6 = (1, 2)_d}\), and~\({11 = (2, 3)_d}\).
     For an example, the arc corresponding to point~\({3 = (0, 3)_d}\) is omitted since~\({12 - 3 =
     9 > 4 = \frac{n}{4}}\) where~\({12 = (3, 0)_d}\).
   }
   \label{fig:circ-arc-lb-construction}
 \end{figure}
 
 \begin{lemma}
 \label{lem:ArcsEndingPoints}
  Let $s$ be an integral arc of $C$ of length at least $d - 1$.
  Then, it contains at least $\frac{d}{5}$ integer points such that they are starting points of arcs in $R$.
  Similarly, it contains at least $\frac{d}{5}$ integer points such that they are ending points of arcs in $R$.
 \end{lemma}
 \begin{proof}
  Since $s$ is integral and its length is at least $d - 1$, it contains at least $d$ consecutive integer points.
  Thus, there is an integer $j$ such that $s$ contains all points of the set $\{(j,t)_d,\dots, (j,d-1)_d, (j+1, 0)_d, \dots, (j+1,t-1)_d\}$ for some $t \in [d]$.
  For any $\ell \in [d]$, we define $j_\ell = j$ if $\ell \ge t$ and $j_\ell = j+1$ otherwise.
  It follows that $s$ contains the point $(j_\ell, \ell)_d$ for every $\ell \in [d]$.
  Consider a subset of these points $S = \left\{\bigl(j_k, k \bigr)_d \mid 1 \le (k - j \mod d ) \le \frac{d}{5} \right\}$.
  Observe that $|S| = \frac{d}{5}$.
  We claim that each point of $S$ is a starting point of an arc in $R$.
  Let $r$ be an arc $\bigl[(j_k,k)_d, (k, j_k)_d\bigr]$ where $1 \le (k - j \mod d) \le \frac{d}{5}$.
  Let $r_0 = \bigl[(j_k,0)_d, (k+1, 0)_d\bigr]$.
  Note that $r$ is a subarc of $r_0$.
\begin{align*}
	|r_0| &= (k + 1) \cdot d - j_k \cdot d \mod n \\
	 &= (k + 1 - j_k \mod d) \cdot d \tag{since $n = d^2$} \\
	 &\leq (k - j \mod d) \cdot d + d \tag{since $j \leq j_k \leq j + 1$} \\
	& \leq \frac{d}{5} \cdot d + d \leq \frac{n}{4} \tag{for sufficiently large $n$}   
\end{align*}
  Thus, $r \in R$.
  Analogously, the set $E = \left\{\bigl(j_k, k \bigr)_d \mid 1 \le (j + 1 - k \mod d ) \le \frac{d}{5} \right\}$ of size $\frac{d}{5}$ contains ending points of arcs in $R$.
 \end{proof}

 \begin{lemma}
 \label{lem:LongSubarc}
  Let $r = [(i,j)_d, (j,i)_d]$ and $r' = [(i',j')_d, (j',i')_d]$ be two arcs in $R$.
  Then, at least one of the arcs $r$ and $r'$ contains an integral subarc $s$ of length $d-2$ that is disjoint from the other arc.
 \end{lemma}
 \begin{proof}
  If the arcs $r$ and $r'$ are disjoint then the existence of $s$ is trivial as each arc in $R$ has a length at least $d - 1$ (the arcs in $R$ of length exactly $d - 1$ are those of the form $[(k, k+1)_d, (k+1,k)_d]$).
  
  Thus, suppose that $r$ and $r'$ intersect and without loss of generality suppose that the ending point of $r$ lies inside $r'$, i.e. we read in clockwise order the points $(i',j')_d$, $(j,i)_d$ and $(j',i')_d$.
  Consider an arc $s' = [(j,i)_d, (j',i')_d]$.
  Note that arc $s'$ is a subarc of $r'$ and intersects $r$ only in the point $(j,i)_d$.
  Thus, if $|s'| \geq d - 1$, the arc $r'$ would contain the sought integral subarc $s$.
  We have $|s'| = (j' - j)\cdot d + i' - i \pmod n$.
  Since $0 \leq i,i',j,j' \leq d-1$, it holds that if $|s'| \leq d - 2$ then $j = j'$ or $j' - j = 1 \pmod d$ and $i - i' \geq 2$ (note that this holds even when $j = d-1$ and $j' = 0$).

  First, suppose that $i - i' \geq 2$.
  Then, the arc $s'' = [(i'+1,0)_d, (i,0)_d]$ is a subarc of $r'$ and it might intersect the arc $r$ only in the point $(i,j)_d$.
  Since $i'+1 < i$ by the assumption, the length of $s''$ is at least $d$ and the lemma follows for this case.
  
  The last case is when $j' = j$.
  This implies that $i \neq i'$ as all endpoints are unique.
  Moreover, since the point $(j,i)_d$ precedes the point $(j',i')_d$ in clockwise order, we get that  $i < i'$.
  Then, the arc $s'' = [(i,j)_d, (i',j')_d] = [(i,j)_d, (i',j)_d]$ has length at least $d$ and is a subarc $r$.
  Clearly, the arc $s''$ intersects with $r'$ only in the point $(i',j')_d$. 
 \end{proof}
 
 Now, we are ready to prove Theorem~\ref{thm:CircArcLB}.
 \begin{proof}[Proof of Theorem~\ref{thm:CircArcLB}]
  Consider two arcs $r = [(i,j)_d, (j,i)_d]$ and $r' = [(i',j')_d, (j',i')_d]$ in $R$.
  Recall that all arcs in $R$ have length at most $\frac{n}{4}$.
  Thus, we can suppose that the arc $s' = [(j',i')_d, (i,j)_d]$ has length at least $\frac{n}{4}$ and is disjoint from $r$ and $r'$, except for the endpoints $(j',i')_d$ and $(i,j)_d$.
  Note that, this implies that the endpoints $(j,i)_d$ and $(j',i')_d$ of $r$ and $r'$, respectively, are in clockwise order.
  By Lemma~\ref{lem:LongSubarc}, we suppose that $r$ contains a subarc $s$ of length $d - 2$ that is disjoint with $r'$ (the other case when $s$ is a subarc of $r'$ is analogous).
  Let $L$ be a set of points in $s$ such that they are ending points of arcs in $R$ (distinct from $r$).
  By Lemma~\ref{lem:ArcsEndingPoints}, we have that $|L| \geq \frac{d}{5} - 1$.
  
  Let $t = [a, b]$ be an arc with the ending point $b$ in $L$.
  The arc $t$ intersects the arc $r$ as the ending point $b$ is a point of $s \subseteq r$.
  Since the arc $s'$ has length at least $\frac{n}{4}$ and $t \in R$, the starting point $a$ has to be a point of $s'$ and thus, the arc $t$ is disjoint from $r'$.
  Therefore, $\sd(r, r') \geq |L| \geq \Omega(\sqrt{n})$.
 \end{proof}

 \subsection{Interval Graphs}
 In this section, we will prove that the symmetric difference of interval graphs is still a fixed power of $n$ but strictly less than the symmetric difference of circular arc graphs (that is $\Theta(\sqrt{n})$).
 In particular, we will prove the following two theorems.
 
 \begin{theorem}
 \label{thm:IntervalUB}
 Any interval graph $G \in \INT_n$ has symmetric difference at most $O(\sqrt[3]{n})$.
 \end{theorem}
 
 \begin{theorem}
 \label{thm:IntervalLB}
 There is an interval graph $G \in \INT_n$ of symmetric difference at least $\Omega(\sqrt[4]{n})$.
 \end{theorem}

 Note that the existence of interval graphs of arbitrarily high symmetric difference is proved in Corollary~5.3 of Dallard et al.~\cite{box_graphs}. While they do not provide explicit bounds, their proof gives the same $\Omega(\sqrt[4]{n})$ bound as ours. However, our proof of the lower bound is self-contained and we believe it to be simpler. 
 We start with a proof of the upper bound.
 \begin{proof}[Proof of Theorem~\ref{thm:IntervalUB}]
 Let $G = (V,E)$ be an intersection graph of intervals $R = \{r_1,\dots,r_n\}$ with $r_i = [a_i, b_i]$ and let $V = \{1,\dots,n\}$.
 Without loss of generality, we suppose for clarity that all $a_i$'s and $b_i$'s are different points.
 The intervals are numbered in the order given by their starting points, i.e., for every two indices $i,j \in [n]$ we have $i < j$ if and only if $a_i < a_j$.
 For an interval $r_i$, we define two sets:
 \begin{enumerate}
 \item $A_i = \{j \mid a_i < b_j \}$, i.e., it contains the indices of intervals in $R$ that end after the interval $r_i$ starts.
   \item $B_i = \{j \mid a_j < b_i \}$, i.e., it contains the indices of intervals in $R$ that start before the interval $r_i$ ends.
 \end{enumerate}  
 Note that $N(i) = A_i \cap B_i$.
 Moreover, $\sd(i, j) \leq|N(i)\sym N(j)| \leq |A_i \sym A_j| + |B_i \sym B_j|$.
 Note that for each pair $i,j$, it holds that $A_i \subseteq A_j$ or $A_j \subseteq A_i$, thus $A_i \sym A_j$ is $A_i \setminus A_j$ or $A_j \setminus A_i$ and analogously with $B_i$ and $B_j$.
 We will prove that there are two intervals $r_i$ and $r_j$ such that $|A_i \sym A_j| + |B_i \sym B_j| \leq O(\sqrt[3]{n})$.
 
 Let $d \in \N$ be a parameter.
 We will find two vertices $i,j \in V$ such that $\sd(i,j) \leq O(\max \{d, d + \frac{n}{d^2}\})$.
 Thus, if we set $d = \sqrt[3]{n}$ we would get $\sd(i,j) \leq O(\sqrt[3]{n})$.
 Since any subgraph of an interval graph is again an interval graph, the upper bound for $\sd(G)$ will follow.
 Let $D_\ell = A_\ell \setminus A_{\ell + 1}$.
 In other words, the set $D_\ell$ contains intervals of $R$ that end after the interval $r_\ell$ starts but before the interval $r_{\ell + 1}$ starts.
 Note that $B_i = \{1,\dots,\ell\}$ where $\ell$ is the unique index such that $i \in D_\ell$.
 Let $k$ be the largest index such that $\sum_{\ell \leq k} |D_\ell| \leq d + 2$.
 Note that for any two indices $i,j \leq k$, it holds that $|A_i \sym A_j| \leq \sum_{\ell \leq k} |D_\ell|$.
 
 First, we prove that if $k \leq d^2$, then $\sd(i,j) \leq 2d + 2$.
 Suppose that actually $\sum_{\ell \leq k} |D_\ell| \leq d$.
 Then, $|D_{k + 1}| \geq 3$.
 Let $i,j \in D_{k + 1}$ such that $i,j \leq k$. 
 Then, $B_i = B_j$ and $|A_i \sym A_j| \leq d$, therefore $\sd(i,j) \leq d$.
 
 From now, we suppose that $d \leq \sum_{\ell \leq k} |D_\ell| \leq d + 2$.
 Let $p \leq k$ be an index such that $|D_p| \geq 2$ (if such $p$ exists) and $i,j \in D_p = A_p \setminus A_{p+1}$.
 Thus, $B_i = B_j$.
 Since $i,j \leq p \leq k$, then $|A_i \sym A_j| \leq d+2$ and $\sd(i,j) \leq d+2$.
 Note that so far, we did not use the assumption that $k \leq d^2$.

 Now, suppose that for all $\ell \leq k$ it holds that $|D_\ell| \leq 1$.
 Since $k \leq d^2$ there exists two indices $p < q \leq k$ such that $D_p, D_q \neq \emptyset$ and $q - p \leq d$ (there are at least $d$ indices $\ell \leq k$ such that $|D_\ell| = 1$).
 Let $i \in D_p$ and $j \in D_q$. 
 Since $B_i = \{1, \dots, p\}, B_j = \{1, \dots, q\}$, we have $|B_j \setminus B_i| \leq d$.
 Further, since $i, j \leq k$, we have $|A_i \sym A_j| \leq d+2$.
 Thus, $\sd(i,j) \leq 2d + 2$.

 Now, we suppose that $k > d^2$.
 It follows there are two indices $i,j \leq k$ such that $i \in D_p$ and $j\in D_q$ such that $|p - q| \leq \frac{n}{d^2}$.
 Therefore, $|A_i \sym A_j| \leq d + 2$ and $|B_i \sym B_j| = |p - q| = \frac{n}{d^2}$ and $\sd(i,j) \leq d + 2 + \frac{n}{d^2}$.
  \end{proof}

Now, we give a construction of an interval graph with symmetric difference $\Omega(\sqrt[4]{n})$.
Let $d \in \N$ sufficiently large.
We construct $\Theta(d^4)$ intervals on a line segment $[0, t]$ for $t = 20d^3$ such that the corresponding intersection graph $G$ will have $\sd(G) \geq d$.
There will be intervals of two types -- short and long.
See Figure~\ref{fig:interval-graph-lb-construction} for an illustration.
Short intervals have length $d$ and they start in each point $0,\dots,t-d$, i.e.,
\[
S = \bigl\{[i, i + d] \mid i \in \{0,\dots,t-d\}\bigr\}.
\]
Long intervals will have various lengths.
For $i \geq 0$, let $\ell_i = 4d^2\cdot(i + 1)$.
For $0 \leq i \leq 2d-1$, we define the $i$-th class of long intervals as 
\[
L_i = \bigl\{[a, b] \mid a,b \equiv i \pmod{2d}; \ell_i \leq b - a \leq \ell_i + 2d^2 \bigr\},
\]
i.e., the set $L_i$ contains intervals such that they start and end in points congruent to $i$ modulo~$2d$ and their length is between $\ell_i$ and $\ell_i + 2d^2$.
Let $I = S \cup \bigcup_{0 \leq i \leq 2d-1} L_i$ be the set of all constructed intervals.
We start with two observations about $I$.

\begin{figure}
  \centering
  \includegraphics[width=\textwidth]{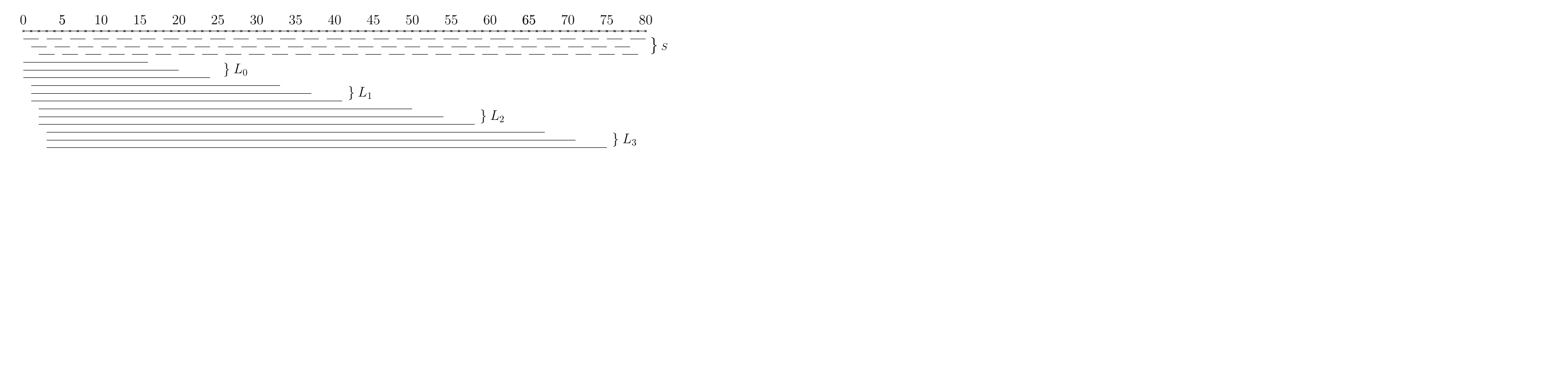}
  \caption{%
    An example of the interval graph lower bound construction for~\(d = 2\).
    For clarity, only the first three intervals are displayed from each set $L_i$. 
  }
  \label{fig:interval-graph-lb-construction}
\end{figure}

\begin{observation}
\label{obs:ShortIntervals}
Any interval $[a,b]$ (for $a,b \in \N$) of the line segment $[0,t]$ of length $2d$ contains $d$ short intervals.
\end{observation}

\begin{observation}
\label{obs:IntervalCount}
There are at most $O(d^4)$ intervals in $I$.
\end{observation}
\begin{proof}
Clearly, there are $t - d + 1 = O(d^3)$ intervals in $S$.
Note that for any $a \leq \frac{t}{2}, a \equiv i \pmod{2d}$, there are exactly $d + 1$ intervals of a form $[a,b]$ in $L_i$ because of the length constraints of the long interval.
Analogously, for any $b \geq \frac{t}{2}, b \equiv i \pmod{2d}$, there are exactly $d + 1$ intervals of a form $[a,b]$ in $L_i$.
We remark this indeed holds even for $L_{2d-1}$ as we set $t$ to be large enough.
There are no other intervals in $L_i$.
Since there are $O(d^2)$ points $p \in [0,t]$ such that $p \equiv i \pmod{2d}$, it follows that $|L_i| \leq O(d^3)$.
Therefore, there are at most $O(d^4)$ long intervals as there are $O(d)$ classes of long intervals.
\end{proof}

The graph $G$ is an intersection graph of $I$.
Now, we are ready to prove Theorem~\ref{thm:IntervalLB}, i.e., $\sd(G) \geq \Omega(\sqrt[4]{n})$.

\begin{proof}[Proof of Theorem~\ref{thm:IntervalLB}]
Let $r = [a,b]$ and $r' = [a', b']$ be two intervals in~$I$.
Without loss of generality let $a \leq a'$.
First, suppose that the $r,r' \in S$, i.e., both of them are short.
In this case, it holds that $a < a'$ and $b < b'$.
First, if $a' > a + 2d$, then all $d$ short interval of form $[i, i + d]$ for $i \in r$ do not intersect $r'$.
Thus, $|N(r) \sym N(r')| \geq d$.

Now, suppose that $a' \leq a + 2d$. 
Further, suppose that $b' \leq \frac{t}{2}$.
Then as already observed, there are $d$ intervals of the form $[b', c]$ in $L_i$ for $b' \equiv i \pmod{2d}$.
Since $b'$ is not in $r$, we have that $|N(r) \sym N(r')| \geq d$.

If $b' > \frac{t}{2}$, then $a > \frac{t}{2} - 3d$, as $b' - d = a' \leq a + 2d$.
Analogously, it holds there are $d$ intervals of the form $[c, a]$ in $L_i$ for $a \equiv i \pmod{2d}$ as $d$ and $t$ is large enough.

Now, suppose that $r, r' \in L_i$ for some $i$.
In this case $a,b,a',b' \equiv i \pmod{2d}$. 
Since $r \neq r'$, it follows that $|a - a'| \geq 2d$ or $|b - b'| \geq 2d$.
Thus, at least one of the intervals $r$ and $r'$ has a private subinterval of length at least $2d$, and by Observation~\ref{obs:ShortIntervals}, we have that $|N(r) \sym N(r')| \geq d$.

Let $k$ be the difference of length of $r$ and $r'$.
For the remaining cases, we will prove that $k \geq 4d$.
Then, at least one of the intervals $r$ and $r'$ contains a private subinterval of length at least $2d$, and again by Observation~\ref{obs:ShortIntervals}, we have that $|N(r) \sym N(r')| \geq d$.
There are two remaining cases:
\begin{enumerate}
\item $r \in S, r' \in L_i:$ Then, $|r| = d$ and $|r'| \geq 4d^2$.
\item $r \in L_i, r' \in L_j$ for $i < j$: Then, $|r| \leq 4d^2\cdot(i + 1) + 2d^2$ and $|r'| \geq 4d^2\cdot (j + 1)$. It follows that $k = 4d^2\cdot (j - i) - 2d^2 \geq 2d^2$.
\end{enumerate}
Thus, in both cases, we have that $k \geq 4d$ for $d \geq 2$.
We have shown that $|N(r) \sym N(r')| \geq d$ for all cases.
Thus by Observation~\ref{obs:IntervalCount}, we conclude that $\sd(G) \geq \Omega(\sqrt[4]{n})$.
\end{proof}

\section{Open Questions}
We leave two gaps in the bounds provided by this paper.
\begin{enumerate}
 \item $O\bigl(\sqrt{\log n}\bigr)$-gap between the lower and upper bound for the functionality of all graphs as we showed that $\Omega\bigl(\sqrt{n}\bigr) \leq \fun(\mathcal{G}_n) \leq O\bigl(\sqrt{n \log n}\bigr)$ (for $\mathcal{G}_n$ being the class of all graphs on $n$ vertices).
 \item $O\big(\sqrt[12]{n}\bigr)$-gap between the lower and upper bound for the symmetric difference of the interval graphs as we showed that $\Omega\bigl(\sqrt[4]{n}\bigr) \leq \sd(\INT_n) \leq O\bigl(\sqrt[3]{n}\bigr)$.
\end{enumerate}
Thus, two questions arise.
\begin{question}
 What is the true value of the functionality of all graphs on $n$ vertices?
\end{question}

\begin{question}
 What is the true value of the symmetric difference of the interval graphs on $n$ vertices?
\end{question}

The classes of interval graphs and of circular arc graphs are related.
Dallard et al.~\cite{box_graphs} showed that the functionality of the interval graphs is bounded and the symmetric difference of the interval graph is unbounded.
We showed that the symmetric difference of the circular arc graph is unbounded as well, however different from the symmetric difference of the interval graphs.
Thus, a natural question arises.
\begin{question}
 Is the functionality of the circular arc graphs bounded?
\end{question}

\subsubsection*{Acknowledgements}
The research presented in this paper was started during the KAMAK workshop in 2021. We are grateful to the organizers
of this wonderful event.

%

 \bibliographystyle{plain}
 \bibliography{main}

\appendix
\section*{Appendix}
\subsection*{Poisson Approximation of Balls and Bins}
In this section, we prove Theorem~\ref{thm:PoisApx}.
Let us recall the setting.
There are $m$ balls and $n$ bins.
Each ball is thrown into the $i$-th bin with a probability $p_i$, independently of the other throws (and $\sum_{i \in [n]} p_i = 1$).
Let $X_i$ be the number of balls in the $i$-th bin after all balls are thrown.
Let $Y_i$ by a random variable of Poisson distribution such that $\E[Y_i] = p_i m$ and is independent of other variables $Y_j$'s.
We will prove the following theorem.
\begin{theorem}
\label{thm:PoisApxFunc}
Let $f: \R^n \to \R$ be a non-negative function.
Then,
\[
\E[f(X_1,\dots,X_n)] \leq e^{1/12m} \sqrt{2\pi m}\cdot \E[f(Y_1,\dots,Y_n)].
\]
\end{theorem}

Theorem~\ref{thm:PoisApx} follows from the previous theorem when we set the function~$f$ to be the indicator function of the event $A$.
Let $Y = \sum Y_i$.
Since $Y_i$'s are independent, the variable $Y$ has Poisson distribution such that $\E[Y] = \sum \E[Y_i] = m$.
First, we start with an auxiliary lemma.

\begin{lemma}
\label{lem:PoisApx}
The distribution of $(Y_1,\dots,Y_n)$ conditioned on $Y = m$ is the same as $(X_1,\dots,X_n)$.
\end{lemma}
\begin{proof}
Let $m_1,\dots,m_n$ be non-negative integers such that $\sum m_i = m$.
When throwing $m$ balls into $n$ bins we have the following,
\[
\Pr[(X_1,\dots,X_n) = (m_1,\dots,m_n)] = \frac{m!}{\prod_{i \in [n]} m_i!} \cdot \prod_{i \in [n]} p_i^{m_i}.
\]
Now, for the Poisson variables.
Since $Y_i$'s are independent, we have the following.
\begin{align*}
\Pr&\bigl[(Y_1,\dots,Y_n) = (m_1,\dots,m_n) \mid Y = m\bigr] = \frac{\prod \Pr[Y_i = m_i]}{\Pr[Y = m]} \\
&= \frac{\prod e^{-p_i m}(p_i m)^{m_i} / m_i !}{e^{-m}m^m/m!} = \frac{m!}{\prod_{i \in [n]} m_i!} \cdot \prod_{i \in [n]} p_i^{m_i}
\end{align*}
\end{proof}

Further, for the proof of Theorem~\ref{thm:PoisApxFunc}, we use the well-known Stirling's formula (see for example the book by Mitzenmacher and Upfal~\cite{prob_comp_book}).

\begin{lemma}[Stirling's formula]
 \[
  \sqrt{2\pi n} \left(\frac{n}{e}\right)^n \leq n! \leq e^{1/12n} \cdot \sqrt{2\pi n} \left(\frac{n}{e}\right)^n
 \]

\end{lemma}

\begin{proof}[Proof of Theorem~\ref{thm:PoisApxFunc}]
We have the following.
\begin{align*}
\E&\bigl[f(Y_1,\dots,Y_n)\bigr] = \sum_{k = 0}^{\infty} \Pr[Y = k] \cdot \E\bigl[f(Y_1,\dots,Y_n) \mid Y = k \bigr] \\
&\geq \Pr[Y = m] \cdot \E\bigl[f(Y_1,\dots,Y_n) \mid Y = m \bigr] \tag{since $f$ is non-negative} \\
&=\frac{e^{-m}m^m}{m!} \cdot \E\bigl[f(X_1,\dots,X_n) \bigr] \tag{by Lemma~\ref{lem:PoisApx}} \\
&\geq \frac{1}{e^{1/12m}\sqrt{2\pi m}} \cdot \E\bigl[f(X_1,\dots,X_n) \bigr] \tag{by Stirling's formula}
\end{align*}
\end{proof}

\subsection*{Bounds for Binomial Coefficient}
\BinomialBound*

\begin{proof}
Let $k = \lfloor \alpha n \rfloor$.
Using Stirling's formula stated in the previous section, we have the following two inequalities.
\begin{align*}
\binom{n}{k} &\leq  \frac{e^{\frac{1}{12n}} \cdot \sqrt{2\pi n}  \left(\frac{n}{e}\right)^n}{2\pi\sqrt{k(n-k)}\left(\frac{k}{e}\right)^k \left(\frac{n-k}{e}\right)^{n-k}} \tag{by Stirling's formula}\\
&\leq \frac{n^n}{k^k (n-k)^{n-k}} \tag{since $e^{1/12n}\cdot \sqrt{n} < \sqrt{2\pi k(n-k)}$ for all $1 \leq k \leq \frac{n}{2}$ and $n \geq 2$}\\
&=\exp\left(-k \cdot \ln \frac{k}{n} - (n - k) \cdot \ln \frac{n-k}{n}\right) \\
&=  \exp\left(n \cdot\left(-\frac{k}{n} \cdot \ln \frac{k}{n} - \Bigl(1 - \frac{k}{n}\Bigr) \cdot \ln \Bigl(1 - \frac{k}{n}\Bigr) \right)\right) = \exp\left(n \cdot H\left(\frac{k}{n}\right)\right)\\
&\leq e^{H(\alpha) n} \tag{since $H(x)$ is increasing for $0 < x \leq \frac{1}{2}$, and $k \leq \alpha n \leq \frac{n}{2}$}\\
& \\
\binom{n}{k} &\geq \frac{\sqrt{2\pi n}\cdot \left(\frac{n}{e}\right)^n}{e^{\frac{1}{6n}}2\pi\sqrt{k(n-k)}\cdot \left(\frac{k}{e}\right)^k \left(\frac{n-k}{e}\right)^{n-k}} \tag{by Stirling's formula}\\
&\geq \frac{1}{e^{\frac{1}{6}} \cdot \sqrt{2\pi n}} \cdot \frac{n^n}{k^k (n-k)^{n-k}}  \tag{since $n \geq \sqrt{k(n-k)}$ for all $0 \leq k \leq \frac{n}{2}$}\\
&\geq \frac{1}{e^{\frac{1}{6}} \cdot \sqrt{2\pi n}} \cdot \exp\Bigl(n \ln n - \alpha n \ln (\alpha n) - \bigl((1-\alpha) n + 1\bigr) \ln \bigl((1-\alpha) n + 1\bigr)\Bigr)
\tag{since $k^k$ is increasing, $(n-k)^{n-k}$ is decreasing (in $k$) and $\alpha n - 1 \leq k \leq \alpha n$} \\
&\geq \frac{1}{e^{\frac{1}{6}} \cdot \sqrt{2\pi n}} \cdot \exp\Bigl(n \ln n - \alpha n \ln (\alpha n) - \bigl((1-\alpha) n \bigr) \ln \bigl((1-\alpha) n\bigr) - \ln n - 2\Bigr) \tag{by $\ln (x + 1) \leq \frac{1}{x} + \ln x$} \\
&\geq \frac{1}{22\cdot n^{3/2}} \cdot e^{H(\alpha) n} \tag{since $\sqrt{2 \pi} e^{\frac{13}{6}} < 22$}
\end{align*}
\end{proof}

\subsection*{Technical Inequalities}
In this section, we prove two technical inequalities that appear in the proof of Theorem~\ref{thm:RandomUpperBound}.
These inequalities are upper and lower bounds for certain expressions in $p$ for $p \in (0, 0.3]$.
The proofs of both bounds are similar. 
We will show that these expressions are increasing, or decreasing, respectively, and thus the expression is upper- (or lower-) bounded by the value at $p = 0.3$.
We recall that in the following propositions, for $x \in (0,1)$ the function $H(x)$ is defined as
\[
H(x) = -x \cdot \ln x - (1-x) \cdot \ln (1 - x).
\]

\begin{proposition}
\label{prp:TechnicalForLambdaW}
If $p \in (0, 0.3]$, then
\[
\frac{5}{4} - \frac{45p}{40} - \frac{9}{10} \ln p - \frac{H(1.1p)}{p} \leq -\frac{1}{10}.
\]
\end{proposition}
\begin{proof}
Let $f(p) = \frac{5}{4} - \frac{45p}{40} - \frac{9}{10} \ln p - \frac{H(1.1p)}{p}$.
We will show that $f(p)$ is increasing on the interval $(0, 0.3]$.
Since $f(0.3) < -0.11$, it will show the proposition.
To do so, we will show the first derivation of $f$ is positive on the interval $(0,0.3]$.
\begin{equation}
\label{eq:FirstDerivativeF_1}
f'(p) = -\frac{45}{40} - \frac{9}{10p} - \frac{\ln (1 - 1.1p)}{p^2} = - \frac{45p^2 + 36p + 40\ln (1 - 1.1p)}{40p^2}
\end{equation}
Thus, $f'(p) > 0$ if and only if the numerator of (\ref{eq:FirstDerivativeF_1}) is negative.
\begin{align*}
45p^2 + 36p + 40\ln (1 - 1.1p) &\leq 45p^2 + 36p - 44p - 20p^2 \tag{since $\ln(1 - x) \leq -x -\frac{x^2}{2}$ by Taylor expansion of $\ln(1-x)$} \\
&= 25p^2 - 8p < 0 \tag{for $0 < p \leq 0.3$}
\end{align*}
\end{proof}

\begin{proposition}
\label{prp:TechnicalForW}
If $p \in (0, 0.3]$, then
\[
2 - 2(p + \ln p) - \frac{H(1.1p)}{p} \geq 1.
\]
\end{proposition}
\begin{proof}
Let $f(p) = 2 - 2(p + \ln p) - \frac{H(1.1p)}{p}$.
We will show that $f(p)$ is decreasing on the interval $(0, 0.3]$.
Since $f(0.3) > 1.69$, it will show the proposition.
To do so, we will show the first derivation of $f$ is negative on the interval $(0,0.3]$.
\begin{equation}
\label{eq:FirstDerivativeF_2}
f'(p) = -2 - \frac{2}{p} - \frac{\ln (1 - 1.1p)}{p^2} = - \frac{2p^2 + 2p + \ln (1 - 1.1p)}{p^2}
\end{equation}
Thus, $f'(p) < 0$ if and only if the numerator of (\ref{eq:FirstDerivativeF_2}) is positive.
\begin{align*}
2p^2 + 2p + \ln (1 - 1.1p) &\geq 2p^2 + 2p - \frac{33p}{20} \tag{since $1 - x \geq e^{-3x/2}$ for $0 < x \leq \frac{1}{2}$} \\
&= 2p^2 + \frac{7p}{20} > 0
\end{align*}
\end{proof}

\end{document}